\documentclass[10pt] {article}

\usepackage{amsmath}
\usepackage{amssymb}
\usepackage{amsthm}
\usepackage{color}
\usepackage{graphics}
\usepackage{graphicx}
\usepackage[inline]{enumitem}

\newcommand {\A}  {\ensuremath{\operatorname{Aut}}}
\newcommand {\Fix}  {\ensuremath{\operatorname{Fix}}}
\newcommand {\D}    {\ensuremath{\operatorname{D}}}
\newcommand {\Prob} {\ensuremath{\operatorname{Pr}}}
\newcommand{\Sym}{\ensuremath{\operatorname{Sym}}}
\newcommand{\N}{{\mathbb{N}}}

\newcommand {\dsc}  {\ensuremath{\operatorname{DSC}}}
\newcommand {\Vblue}  {\ensuremath{\mathcal{B}}}
\newcommand {\Vred}  {\ensuremath{\mathcal{R}}}

\newtheorem{theorem}{Theorem}
\newtheorem{lemma}[theorem]{Lemma}

\newtheorem{corollary}[theorem]{Corollary}

\theoremstyle{definition}

\newtheorem{example}[theorem]{Example}

\theoremstyle{remark}

\begin{document}

\title
{Distinguishing density and the Distinct Spheres Condition}
\author{
Wilfried Imrich\\
Montanuniversit\"{a}t Leoben, 8700 Leoben, Austria \\
imrich@unileoben.ac.at\\
\and
Florian Lehner\thanks{Florian Lehner was supported by the Austrian Science Fund (FWF) Grant no. J 3850-N32}\\
Mathematics Institute,
University of Warwick,\\
Coventry, UK\\
mail@florian-lehner.net
\and
Simon M.~Smith\\
School of Mathematics and Physics,
University of Lincoln,\\
Lincoln, UK\\
sismith@lincoln.ac.uk
}
\date{\today}
\maketitle

Keywords: {Distinguishing number, automorphisms, infinite graphs.}
{05C25, 05C63, 05C15, 03E10.}

\begin{abstract}

If a graph $G$ has distinguishing number 2, then there exists a partition of its vertex set into two parts, such that no nontrivial automorphism of $G$ fixes setwise the two parts. Such a partition is called a 2-\emph{distinguishing coloring} of $G$, and the parts are called its \emph{color classes}. If $G$ admits such a coloring, it is often possible to find another in which one of the color classes is sparse in a certain sense. In this case we say that $G$ has 2-\emph{distinguishing density zero}. An extreme example of this would be an infinite graph admitting a 2-distinguishing coloring in which one of the color classes is finite.

If a graph $G$ contains a vertex $v$ such that, for all $n \in \N$, any two distinct vertices equidistant from $v$ have nonequal $n$-spheres, then we say that $G$ satisfies the \emph{Distinct Spheres Condition}. In this paper we prove a general result: any countable connected graph that satisfies the Distinct Spheres Condition is 2-distinguishable with density zero. We present two proofs of this, one that uses a deterministic coloring, and another (that applies only to locally finite graphs) using a random coloring. From this result, we deduce that several important families of countably infinite and connected graphs are 2-distinguishable with density zero, including those that are locally finite and primitive. Furthermore, we prove that any connected graph with infinite motion and subquadratic growth is 2-distinguishable with density zero.

\end{abstract}

\section{Introduction}

Let $G$ be a graph without loops or multiple edges.
If $G$ admits an $n$-coloring of its vertices that is only preserved by the identity automorphism, then $G$ is called \emph{$n$-distinguishable} and the $n$-coloring is called an \emph{$n$-distinguishing coloring}.

A large number of papers have been written about the distinguishability of finite and infinite graphs. Arguably, in this context the most important class of graphs are those that are $2$-distinguishable. Such graphs have automorphism groups that are minimal with respect to the relation of strong orbit equivalence (see \cite{bailey-cameron-11} or \cite{imsmtuwa-14} for example), and there are important permutational properties which, when exhibited by $\A(G)$, guarantee that $G$ is $2$-distinguishable. For example, it is known that
every finite vertex-primitive graph with at least $33$ vertices, other than the complete graph or its complement, is 2-distinguishable
(see \cite[Theorem 1]{seress-97}), and the only known proof of this depends on the Classification of the Finite Simple Groups.
A similar result holds for (countably) infinite vertex-primitive graphs with infinite diameter (see \cite[Theorem 1]{smtuwa-xx}). There is a well-known open conjecture about 2-distinguishability, called \emph{The Infinite Motion Conjecture}. A graph $G$ is said to have \emph{infinite motion} if every nontrivial automorphism of $G$ moves infinitely many vertices, and the conjecture states that every connected, locally finite graph with infinite motion is 2-distinguishable.

In this paper we focus on this important class of 2-distinguishable graphs, and consider those that are countably infinite and connected. For such a graph $G$ one can partition the vertex set $V = \Vblue \cup \Vred$ (into ``blue'' and ``red'' color classes of vertices) such that the identity is the only automorphism of $G$ that fixes setwise $\Vblue$ and $\Vred$. We always assume, without loss of generality, that $|\Vblue| \leq |\Vred|$.

It is possible that $|\Vblue|$ is finite, even when $G$ is infinite. In this case the minimum cardinality of $\Vblue$ among all such partitions is called the \emph{2-distinguishing cost} of $G$,
and it is known, for example, that a connected, locally finite graph $G$ whose automorphism group is infinite has finite 2-distinguishing cost if and only if $\A(G)$ is countable (see \cite[Theorem 3.3]{boutin-imrich-14}).
In this paper we do not consider this situation, and instead investigate the more common scenario in which $|\Vblue|$ is infinite for every 2-distinguishing coloring of $G$.
We prove that
it is possible to find a 2-distinguishing coloring of $G$ in which the blue vertices are sparse in a certain sense (see Section~\ref{section:prelim} for a precise definition)
in the following situations:
\begin{enumerate}[label=(\roman*)]
\item
	if $G$ is locally finite and primitive (see Corollary~\ref{main:cor});
\item
	 $G$ is a tree without leaves (see Corollary~\ref{main:cor});
\item
	 $G$ is vertex-transitive of connectivity 1  (see Corollary~\ref{main:cor});
\item
	 $G$ is the Cartesian product of any two connected graphs of infinite diameter  (see Corollary~\ref{main:cor});
\item \label{label:last}
	 $G$ has infinite motion and subquadratic growth (see Theorem~\ref{thm:spheres}); and
\item
	 $G$ satisfies the Distinct Spheres Condition (see Theorem~\ref{thm:dsc}).
\end{enumerate}
Our proof of \ref{label:last} simplifies the proof of the main result of \cite{cuimle-2014},
that locally finite graphs with infinite  motion and subquadratic growth are 2-distinguishable.

\section{Preliminaries}
\label{section:prelim}

Throughout this paper, $G$ will be a simple graph without loops or multiple edges, with vertex set $V$ and edge set $E$. An $X$-\emph{coloring} (sometimes called an $X$-\emph{labeling}) $l$ of $G$ is a mapping $l:V \rightarrow X$. If $X$ has cardinality $n$, we say that $l$ is an $n$-coloring. In this paper, $l$ will typically be a 2-coloring, with $X = \{\text{blue},\text{red}\}$.

Let $l$ be an $X$-labeling of $G$ and $g\in\A(G)$. If, for every $v\in V$, we have that $l({g(v)}) = l(v)$, then we say that $l$ \emph{is preserved by} $g$. If $l$ is not preserved by $g$, then we say that $g$ \emph{breaks} $l$. The coloring $l$ is called \emph{distinguishing} if it is only preserved by the trivial automorphism.  Accordingly, an \emph{$n$-distinguishing coloring} of $G$ is a distinguishing coloring that is also an $n$-coloring.
The \emph{distinguishing number} $\D(G)$ of $G$ is then the least cardinal $d$ such that there exists a $d$-distinguishing coloring of $G$.

The \emph{ball with center $v\in V$ and radius $r$} in $G$ is the set of all vertices $x\in V$ with $d(v,x)\leq r$ and is denoted by $B_{v}(r)$.
The \emph{sphere with center $v\in V$ and radius $r$} is the set of all vertices $x\in V$ with $d(v,x)=r$ and is denoted by $S_{v}(r)$.
For terms not defined here we refer the reader to \cite{haimkl-11}.

The graphs considered in this paper are infinite. If a graph $G$ is locally finite, then all balls and spheres of finite radius are finite. Moreover, if the diameter of $G$ is infinite then, for all $r \in \N$,
\[\left|B_{v}(r)\right| = \sum_{i={0}}^r\left|S_{v}(i)\right|\quad\text{and}\quad\left|S_{v}(i)\right| \geq 1\,.\]
For functions $F : \N \rightarrow \N$ and $f: \N \rightarrow \N$, we write $F(r) \sim \mathcal{O}(f(r))$ if there exists a constant $c$ such that $F(r) \leq c f(r)$ for all $r \in \N$.\\

The graphs in this paper will be 2-distinguishable.
Let $G$ be a countable connected graph that admits a 2-distinguishing coloring $l$.
For any set $W$ of vertices in $G$, and any vertex $v$ of $G$, we define the density of $W$ at $v$ to be
\[\delta_v(W) := \lim_{n\rightarrow \infty}\frac{|B_v(n)\cap W|}{|B_v(n)|}\]
if this limit exists, with the convention that a finite cardinal divided by an infinite cardinal is zero, and a quotient of two infinite cardinals is undefined. If $\delta_v(W)$ exists for all vertices $v \in V$, then we define the density of $W$ to be
\[\delta(W) := \sup \{\delta_v(W) : v \in V\}.\]

If $\Vblue$ and $\Vred$ are the two color classes of $l$ and $\delta(\Vblue)$ or $\delta(\Vred)$ exist, then the \emph{density of $l$ at $v$} is $\delta_v(l) := \min \{\delta_v(\Vblue), \delta_v(\Vred)\}$, and the \emph{density of $l$} is
\[\delta(l) := \min \{\delta(\Vblue), \delta(\Vred)\}.\]

If $l$ is a 2-distinguishing coloring of $G$ and $\delta_v(l) = 0$ for some vertex $v$, then we say that $G$ has \emph{2-distinguishing density zero at $v$}. If  $\delta(l) = 0$ then we say that $G$ has \emph{2-distinguishing density zero}.

\begin{lemma} \label{lemma:at_v_vs_all_v}
Let $G$ be a connected graph and let $v, w$ be vertices in $G$. Suppose there is a constant $c$ such that for all $n \in \N$ we have $|B_w(n+1)| < c \cdot |B_w(n)|$. If $G$ has 2-distinguishing density zero at $v$, then $G$ has 2-distinguishing density zero.
\end{lemma}
\begin{proof} Suppose $l$ is a 2-distinguishing coloring of $G$ with density zero at $v$, for some vertex $v \in V$. Let $\Vblue$ be a color class of $l$ such that $\delta_v(\Vblue) = 0$.
For all vertices $x,y \in V$ and all $n \geq d(x,y)$, we have
$B_y(n - d(x,y))\subseteq B_x(n) \subseteq B_y(n + d(x,y))$.
Hence, for all suitably large $n$,
\begin{equation} \label{eq:exp}
c^{-d(w,y)} \cdot |B_w(n)| \leq |B_y(n)| \leq |B_w(n)| \cdot c^{d(w,y)}.
\end{equation}

Fix any vertex $x \in V$. Now,
$|B_x(n)| \geq c^{-d(w,x)} |B_w(n)|$ by \eqref{eq:exp}, and $|B_w(n)| > c^{-d(v,x)} |B_w(n+d(v,x))|$ by assumption. Moreover,
 by \eqref{eq:exp}, $|B_w(n+d(v,x))| \geq c^{-d(w,v)} |B_v(n+d(v,x))|$. Hence, writing $k := c^{-d(w,x)-d(v,x)-d(w,v)}$ and $d := d(v,x)$, we have
\[|B_x(n)| > k |B_v(n+d)| \quad \text{ and } \quad |B_x(n) \cap \Vblue| \leq |B_v(n+d) \cap \Vblue|.\]
Therefore
\[\frac{|B_x(n) \cap \Vblue|}{|B_x(n)|} \leq \frac{1}{k} \cdot \frac{|B_v(n+d) \cap \Vblue|}{|B_v(n+d)|}.\]
Since the coloring $l$ has zero density at $v$, it follows that the coloring must also have zero density at $x$.
\end{proof}

The following is an example of a connected graph that does not have 2-distinguishing density zero, but nevertheless has  2-distinguishing density zero at some vertex $v$.

\begin{example}
Let $G$ be a graph constructed as follows. Start with two families of disjoint paths, $P_n$ and $P_n'$, where $P_1$ and $P_1'$ are trivial one-vertex paths and $|P_n| = |P_n'| = n \sum_{i=1}^{n-1} |P_i|$ for all $n > 1$. For every $n > 1$ add edges between all vertices of $P_n$ and one endpoint of $P_{n-1}$ and between all vertices of $P_n'$ and one endpoint of $P_{n-1}'$. Add another edge between the unique vertex $v$ of $P_1$ and the unique vertex $v'$ of $P_1'$. Finally, for every $n>1$ and every vertex $x \in P_n'$ we introduce two new vertices $v_1^x$ and $v_2^x$ and connect them to both $x$ and the unique neighbor of $x$ in $P_{n-1}'$.

Every automorphism of $G$ must pointwise fix the set $\left (\bigcup_{i \in \N} P_i \right ) \cup \left ( \bigcup_{i \in \N} P_i' \right )$. Furthermore, for any $x \in \bigcup_{i \in \N} P_i' $ there is an automorphism $\sigma_x$ of $G$ that interchanges $v_1^x$ and $v_2^x$ while fixing all other vertices of $G$. From these observations we can draw two conclusions. The first is that for any 2-distinguishing coloring $l'$, each pair $\{v_1^x, v_2^x\}$ must be assigned different colors. The second is that the following coloring $l$ is 2-distinguishing: $\Vblue = \{ v_1^x : x \in \bigcup_{i \in \N} P_i'\}$ and $\Vred = V(G) \setminus \Vblue$.

We claim that the coloring $l$ has density zero at the unique vertex $v$ of $P_1$. Indeed, the only blue vertices in $B_v(n)$ are vertices $v_1^x$ for $x \in \bigcup_{i = 1}^{n-1} P_i'$, but $B_v(n)$ also contains the red vertices of $P_n$. Hence,
\[
\frac{|B_v(n)\cap \Vblue|}{|B_v(n)|} \leq \frac{\sum_{i=1}^{n-1} |P_i'|}{|P_n|} = \frac 1n,
\]
showing that the coloring has density zero at $v$.

On the other hand we claim that any 2-distinguishing coloring $l'$ does not have density zero at the unique vertex $v'$ of $P_1'$. Indeed, let $\Vblue'$ and $\Vred'$ be the two color classes of $l'$, and note that
the ball $B_{v'}(n)$ consists of $\bigcup_{i = 1}^{n-1} P_i$, $\bigcup_{i = 1}^{n} P_i'$, and the vertices $v_1^x$ and $v_2^x$ for $x \in \bigcup_{i = 1}^{n} P_i'$. Hence
\[
|B_{v'}(n)| = 4 \sum_{i=1}^{n-1} |P_i| + 3 |P_n| = \left ( 3+\frac 4n \right ) |P_n|.
\]
As observed previously, for every $x \in P_n'$ the vertices $v_1^x$ and $v_2^x$ must lie in distinct color classes of $l'$, because $l'$ is distinguishing. Hence $|B_{v'}(n) \cap \Vblue| \geq |P_n|$. It follows immediately that $l'$ does not have density zero at $v'$. \end{example}

\section{Zero density for graphs satisfying the Distinct Spheres Condition}
\label{sec:spheres}

A  graph $G$ is said to satisfy the \emph{Distinct Spheres Condition}, or DSC, if there
exists a vertex $v\in V$ such that, for all distinct
 $u,w \in V$,
$$d(v,u) = d(v, w) \mbox{ implies } S_u(n) \neq S_w(n) \mbox{ for infinitely many } n\in \mathbb{N}.$$
Notice that graphs that satisfy the DSC have infinite diameter, even when they are not locally finite.
Any countable, connected graph satisfying the Distinct Spheres Condition is known to be 2-distinguishable by The Distinct Spheres Lemma in \cite{imsmtuwa-14}.

Let $\dsc(u,w) := \bigcup_{n \in \N} \left ( S_u(n) \Delta S_w(n) \right )$, where $\Delta$ denotes the symmetric difference. If $G$ satisfies the DSC then for any two vertices $u, w$ equidistant from $v$ there are infinitely many $n \in \N$ for which $S_u(n) \Delta S_w(n)$ is nonempty. Hence $\dsc(u,w)$ is infinite and contains vertices of arbitrary distance from $v$. Note that an automorphism in $G$ that fixes a vertex $x \in \dsc(u,w)$ cannot move $u$ to $w$ or vice versa.

\begin{theorem} \label{thm:dsc}
If $G$ is a countable connected graph that satisfies the Distinct Spheres Condition, then $G$ is 2-distinguishable with density zero.
\end{theorem}

\begin{proof} Suppose $G$ satisfies the DSC with respect to the vertex  $v$. We consider all pairs of vertices $\{u, w\}$ that have equal distance from $v$. As the number of such pairs is countable, we enumerate them in some order, say $\{p_1, p_2, \ldots \}$.

We will denote the set of blue vertices in our coloring by $\Vblue$, where $\Vblue$ consists of $v$, two  adjacent vertices $a, b$ of distance 2 and 3 from $v$, and  isolated vertices $x_1 \in \dsc(p_1), x_2 \in \dsc(p_2), \ldots$ such that for all distinct $i,j \in \N$ we have:
\begin{enumerate*}[label=(\roman*)]
\item
	$d(v, x_i) \geq 7 i^2$;
\item
	$x_i$ and $x_j$ are not equidistant from $v$; and
\item
	$x_i$ and $x_j$ are not adjacent.
\end{enumerate*}
All vertices not in $\Vblue$ are colored red.

Any color preserving automorphism $g \in \A(G)$ has to stabilize the edge $ab$ and thus fixes $v$ since it is the only vertex of distance 2 from $ab$. Hence each $x_i \in \Vblue$ is also fixed by $g$, because it is the only blue vertex in its $v$-sphere.

Suppose a vertex $u$ is not fixed by $g$ and write $w := gu$. Then $\{u,w\}$ is a pair of vertices equidistant from $v$ and thus $\{u,w\} = p_i$ for some $i \in \N$. However, $x_i \in \dsc(u,w)$ is fixed by $g$. This is a contradiction. Hence $g$ is the identity and the coloring is distinguishing.

For any vertex $u$ and any $n \in \N$ we have that $B_u(n) \subseteq B_v(n + d(u,v))$, and hence $|B_u(n) \cap \Vblue| \leq |B_v(n + d(u,v)) \cap \Vblue| \leq 3 + \sqrt{\frac{n+d(u,v)}{7}} \sim \mathcal{O}(\sqrt{n})$. Since $|B_u(n)| \geq n$, we have that $G$ has $2$-distinguishing density zero.
\end{proof}

\begin{corollary} \label{main:cor}
The following graphs are 2-distinguishable with density zero:
\begin{enumerate}[label=(\roman*)]
\item
	connected, locally finite, primitive graphs;
\item
	denumerable trees without leaves;
\item
	denumerable vertex-transitive graphs of connectivity 1; and
\item
	the Cartesian product of any two connected denumerable graphs of infinite diameter.
\end{enumerate}
\end{corollary}

\begin{proof}
By \cite[Theorem 3, Theorem 4, Lemma 8 \& Corollary 10]{smtuwa-xx}, the graphs in Corollary~\ref{main:cor} all satisfy the Distinct Spheres Condition.
\end{proof}

We conclude this section with a probabilistic proof of Theorem~\ref{thm:dsc} in the case where $G$ is locally finite. The random zero-density coloring described in the proof is interesting because it is in some sense canonical, and can be applied to any graph. While this will not necessarily yield a distinguishing coloring, results from \cite{lehner-13} suggest that choosing colors at random is often a good way to obtain a distinguishing coloring.

Define a random coloring $l$ as follows. Fix a root $v_0$ and color every vertex at distance $n$ from the root independently from all other vertices blue with probability $p_n$ and red with probability $1-p_n$. Assume that $\lim_{n \to \infty} p_n =0$ and that $p_n$ is monotonically decreasing. Furthermore assume that $\sum _{n \in \mathbb N} p_n = \infty$. The first assumption makes sure that the colouring has density $0$, whereas the latter assumption will be used to show that the colouring is distinguishing with positive probability. Indeed, it follows from the law of large numbers that the above coloring almost surely has density $0$.

To show that the probability of obtaining a distinguishing coloring is positive, we first introduce an equivalence relation that is almost surely preserved by every color preserving automorphism. Let the relation $\sim$ on $V$ be defined by $u \sim w$ if there is some $n \in \mathbb N$ such that $B_u(n) = B_w(n)$. It is not hard to see that this is an equivalence relation. Note that if $u \nsim w$ then $S_u(n) \neq S_w(n)$ for every $n \in \mathbb N$.

\begin{lemma}
\label{lem:random_eqreln}
If $G$ is a connected, locally finite graph then, almost surely, every automorphism which preserves the random coloring $l$ described above, setwise fixes all equivalence classes w.r.t.\ the relation $\sim$.
\end{lemma}

\begin{proof}
Take an automorphism $g \in \A(G)$ which does not setwise fix all equivalence classes. Then there must be vertices $u,v$ such that $u \not \sim v$ and $g u =v$. Hence every such automorphism is contained in a set of the form
\[
\Lambda_{uvw} := \{g \in \A G \mid g u =v, w = g v\}
\]
for some triple $u,v,w$ such that $u \nsim v$.

We claim that we need only show that the following holds for all triples $u,v,w$ with $u \not \sim v$:
\[
\Prob[\exists g \in \Lambda_{uvw}: l \circ g = l] = 0.
\]
Indeed, in this case by $\sigma$-subadditivity of the probability measure we have that
the probability that there exists a color preserving $g \in \A(G)$ that does not setwise fix the equivalence classes is bounded above by
\[\sum_{\substack{u,v,w \\ u\not \sim v}} \Prob[\exists g \in \Lambda_{uvw}: l \circ g = l] = 0.\]

Fix vertices $u, v, w$ such that $u \not \sim v$. It remains for us to show that the probability of finding a color preserving automorphism in $\Lambda_{uvw}$ is $0$.
If $\Lambda_{uvw}$ is empty, then this holds vacuously, so we may assume that $\Lambda_{uvw}$ is nonempty.
 Observe that every automorphism in $\Lambda_{uvw}$ must move $X_n := S_u(n) \setminus S_v(n)$ to the set $Y_n := S_v(n) \setminus S_w(n)$. Clearly $X_n$ and $Y_n$ are disjoint and have the same cardinality. The sets are nonempty because $S_u(n)$ and $S_v(n)$ are distinct and have the same finite cardinality.

In order to have a color preserving automorphism in $\Lambda_{uvw}$ it is necessary that $X_n$ and $Y_n$ receive the same number of blue vertices for every $n$. For each $n \in \N$ choose $x_n \in X_n$ and let $k(n)$ be the distance from $x_n$ to the root vertex $v_0$. Then $n - d(u,v_0) \leq  k(n) \leq n + d(u,v_0)$ and hence $\lim_{n \to \infty} p_{k(n)} = 0$.

Let $n_0$ be such that $p_{k(n)} \leq \frac 12$ for every $n \geq n_0$. To simplify our notation, let $\Vred$ and $\Vblue$ be respectively the set of red vertices, and the set of blue vertices, and denote by $b(S, T)$ the event that $|S \cap \Vblue| = |T \cap \Vblue|$. For $n \geq n_0$ write $X_n^* := X_n \setminus \{x_n\}$. Then we have,
\begin{align*}
\Prob[b(X_n,Y_n)]
	&\leq \Prob[(x_n \in \Vred) \land b(X_n^*,Y_n)] +
\Prob[(x_n \in \Vblue) \land \lnot b(X_n^*, Y_n)] \\
	&= (1-p_{k(n)}) \Prob[b(X_n^*,Y_n)] + (p_{k(n)}) \Prob[ \lnot b(X_n^*,Y_n)]\\
	&\leq (1-p_{k(n)}) \Big( \Prob[b(X_n^*, Y_n)] + \Prob[\lnot b(X_n^*, Y_n)] \Big)\\
	&\leq (1-p_{n + d(u,v_0)}).
\end{align*}

Hence, $\Prob[\lnot b(X_n, Y_n)] \geq p_{n + d(u,v_0)}$. Since $p_n$ is decreasing and $\sum _{n \in \mathbb N} p_n$ diverges we have that for every $r \in \mathbb N$,
\[
\sum_{n \geq \frac{n_0}{r}} \Prob[\lnot b(X_{rn}, Y_{rn})]
\geq \sum_{n \geq \frac{n_0}{r}} p_{rn + d(u,v_0)}
\geq \frac{1}{r} \sum_{n \geq n_0} p_{n + d(u,v_0)}
=\infty.
\]

If $r>d(u,v)$, then the sets $X_{rn} \cup Y_{rn}$ and $X_{rm} \cup Y_{rm}$ are disjoint whenever $n \not = m$. Since disjoint sets are colored independently, we can invoke the Borel-Cantelli Lemma to conclude that $\lnot b(X_{rn}, Y_{rn})$ almost surely happens for infinitely many $n$. Hence there almost surely is no color preserving automorphism in $\Lambda_{uvw}$.
\end{proof}

\begin{theorem}
If a connected, locally finite graph satisfies the DSC, then the random coloring described above
has almost surely zero density and a
nonzero probability of being distinguishing.
\end{theorem}

\begin{proof}
Suppose that the graph $G$ satisfies the DSC with respect to some vertex $u$. We first show that the orbit of $u$ under the automorphism group contains only finitely many vertices which are equivalent with $u$. Assume that this is not the case. Let $w \sim u$ be an arbitrary vertex equivalent to $u$. Then there is $n_0 \in \mathbb N$ such that $B_{u}(n_0) = B_{w}(n_0)$. The orbit of $u$ contains infinitely many vertices equivalent to $u$ and $G$ is locally finite, hence we can find an automorphism $g$ such that $u \sim g u$ and $d(u, g u) > n_0$. Clearly $B_{g u}(n) = B_{g w}(n)$ for every $n > n_0$. This in particular implies that $d(u, g u) = d(u, g w)$ and thus contradicts the DSC with respect to $u$.

Now let $l$ be a random coloring as described above. Recall that by the law of large numbers, the coloring almost surely has density $0$.
Let $F_{u}$ be the event that every color preserving automorphism fixes $u$ and let $D$ be the event that the random coloring is distinguishing.

By Lemma \ref{lem:random_eqreln} every color preserving automorphism almost surely must map the vertex $u$ to one of the finitely many equivalent vertices in the orbit of $u$. Furthermore there is a nonzero probability that among these finitely many vertices $u$ is the only one colored blue. Hence $\Prob (F_{u}) > 0$.

Finally, if every color preserving automorphism fixes $u$, then each such automorphism also fixes distances from $u$. By the DSC any two vertices at the same distance from $u$ are inequivalent. Hence by Lemma~\ref{lem:random_eqreln} we have that $\Prob(D \mid F_{u}) = 1$ and thus $\Prob(D) = \Prob (F_{u}) \cdot \Prob(D \mid F_{u}) > 0$.
\end{proof}

\section{Zero density for graphs with infinite motion and subquadratic growth}
\label{sec:quadratic}

We say that an infinite, locally finite, connected graph $G$ has \emph{polynomial growth} if there is a vertex $v \in V$ and a constant $c \in \N$ such that $\left|B_{v}(n)\right| \sim \mathcal{O}(n^c)$.
It is easy to see that polynomial growth and the constant $c$ is independent of the choice of vertex $v$.
If $c=1$, then we say the growth is \emph{linear} and if $c=2$ we say it is \emph{quadratic}. Observe that the two-sided infinite path has linear growth, and that the growth of the grid of integers in the plane is quadratic.

For $g \in\A(G)$, let $\Fix(g)$ denote the set of vertices fixed by $g$. The {\em motion} of $g$ is $|V \setminus \Fix(g)|$.
If every nontrivial automorphism of a graph $G$ has infinite motion, we say that $G$ has \emph{infinite motion}. For such graphs the following result will be of importance.

\begin{lemma}[{\cite[Lemma 3.5]{lehner-16}}]
\label{lem:main}
Let $G$ be a graph with infinite motion. For every $g \in\A(G)$, the graph obtained by removing $\Fix(g)$ and all incident edges from $G$ has only infinite components.
\end{lemma}

\begin{theorem}
\label{thm:spheres}
Let $G$ be a connected, locally finite graph with infinite motion. If $\epsilon > 0$ and $G$ contains a vertex $v$ for which there exist infinitely many $n\in\mathbb{N}$ such that
\begin{equation}
\label{eq:smallspheres}
\left|S_{v}(n)\right| \leq \frac{2n(1-\epsilon)}{3}\,,
\end{equation}
then the distinguishing number $\D(G)$ of $G$ is  $2$ and the 2-distinguishing density  is zero.
\end{theorem}

\begin{proof}
If $G$ has linear growth, then $\A(G)$ is countable (this was first observed by T.~Tucker, see \cite[Lemma 4.2]{boutin-imrich-14}). Since $G$ has infinite motion, it is known (see \cite[Theorem 3.2]{boutin-imrich-14}) that the distinguishing number $\D(G)$ of $G$ is $2$ and moreover there is a distinguishing coloring of $G$ involving only two colors in which one of the color classes is finite. Since $G$ is infinite, this implies that the 2-distinguishing density of $G$ is zero. Hence we can assume that $G$ has super-linear growth.
Since $G$ has super-linear growth, there are infinitely many vertices in $G$ of degree at least three. Let $w \not \in B_{v}(2)$ be such a vertex.

We first color all vertices red then we color selected vertices blue. We begin by coloring $v$ and all vertices in $B_{w}(1)$ blue. We will not color any other vertex in $B_{v}(2d(v,w)+1)$ blue, nor will there be another blue vertex with three or more blue neighbors. Note then that any color-preserving automorphism of $G$ must fix $B_w(1)$ setwise and fix $v$ pointwise.

Next, we construct a sequence $\{n_i\}_{i \in \N}$ inductively, with $n_1 \epsilon > 2d(v,w)+1$ and $\left | S_{v}(n_1) \right| \leq \frac{2n_1(1-\epsilon)}{3}$, and for all $i > 1$,
\[n_i \epsilon > n_{i-1} \quad \text{ and } \quad \left | S_{v}(n_i) \right| \leq \frac{2n_i(1-\epsilon)}{3}.\]

Let $\ell(n)$ denote the length of the longest chain of subgroups in the symmetric group $\Sym(n)$. It was shown by Cameron, Solomon and Turull (\cite{casot-89}) that $\ell(n) = \left \lceil \frac{3n}{2} \right \rceil - b(n) -1$, where $b(n)$ is the number of ones in the base $2$ expansion of $n$. In particular, this, together with the bound $\left | S_{v}(n_i) \right| \leq \frac{2n_i(1-\epsilon)}{3}$, implies that $\ell(|S_v(n_i)|) < n_i(1-\epsilon) < n_i - n_{i-1}$.

For $n \in \N$, let $K_n$ be the automorphism group induced by $\A(G)_v$ acting on $B_v(n)$. For $H \leq K_n$, let $\pi_n(H)$ be the subgroup of $\Sym(S_v(n))$ induced by $H$.

Notice that if distinct elements $g_1, g_2 \in K_n$ have the same image under $\pi_n$, then there is an automorphism of $G$ that fixes $S_{v}(n)$ pointwise and acts nontrivially on $B_v(n)$. This contradicts Lemma~\ref{lem:main}. Hence $\pi_n$ is injective. Therefore, for all $i \in \N$, the length of the longest chain of subgroups in $K_{n_i}$ is strictly less than $n_i - n_{i-1}$.

We will now define a set of vertices $C_i \subseteq B_v(n_i) \setminus B_v(n_{i-1})$ that will be colored blue. We will do this in such a way that no two vertices in $C_i$ will have the same distance from $v$. If $B_v(n_i)$ is fixed pointwise by $K_{n_i}$, then choose any $x_0 \in S_v(n_i)$. Otherwise, by Lemma~\ref{lem:main} we may choose $x_0 \in S_v(n_i)$ such that $x_0$ is not fixed by $K_{n_i}$. Assume we have chosen $x_0,
\ldots, x_{j-1}$ for some $j \geq 1$. If there exists $x \in B_{v}(n_i - j)$ such that $x$ is not fixed by the pointwise stabilizer $(K_{n_i})_{(x_0, \ldots, x_{j-1})}$ then let $x_j := x$ and note that $(K_{n_i})_{(x_0, \ldots, x_{j})}$ is a proper subgroup of $(K_{n_i})_{(x_0, \ldots, x_{j-1})}$. Otherwise let $C_i := \{x_0, \ldots, x_{j-1}\}$, and note that in this case the pointwise stabilizer $(K_{n_i})_{(C_i)}$ fixes $S_v(n_i - j)$ pointwise, and hence fixes $B_v(n_i - j)$ pointwise. Since the length of the longest chain of subgroups in $K_{n_i}$ is strictly less than $n_i - n_{i-1}$ this process must terminate with $C_i \subseteq B_v(n_i) \setminus B_v(n_{i-1})$. Moreover,
the setwise stabilizer $(K_{n_i})_{\{C_i\}}$ is equal to the pointwise stabilizer $(K_{n_i})_{(C_i)}$ (because elements in $C_i$ lie in different spheres) and hence
$(K_{n_i})_{\{C_i\}}$ fixes $B_v(n_{i-1})$ pointwise.

We now color all vertices in $\bigcup_{i \in \N} C_i$ blue, and claim that this coloring is distinguishing with density zero. Let $\Vblue$ be the set of blue vertices. If $g \in \A(G)$ is a color-preserving automorphism then, as we argued previously, it must fix $v$ and $\bigcup_{i \in \N} C_i$ pointwise. Hence $g$ fixes $B_v(n_i)$ for all $i \in \N$, and is therefore trivial. Our coloring is thus $2$-distinguishing.

Since the growth of $G$ is super-linear,
\[\lim_{n \rightarrow \infty} \frac{n}{|B_v(n)|} = 0.\]
However, for all large enough $n$ there is at most one blue vertex in $S_v(n)$, so $|B_v(n) \cap \Vblue| \sim \mathcal{O}(n)$. Hence the 2-distinguishing density is zero.
\end{proof}

\end{document}